\documentclass[11pt]{amsart}

\usepackage[margin=1.15in]{geometry}
\usepackage{amscd,amssymb, mathtools, amsmath, setspace, wasysym}
\usepackage{graphicx}
\usepackage[all]{xy}
\usepackage{url}
\usepackage{hyperref}%make all the references and links clickable

\theoremstyle{plain}
\newtheorem{theorem}{Theorem}[section]

\newtheorem{cor}[theorem]{Corollary}

\newtheorem{lemma}[theorem]{Lemma}

\theoremstyle{definition}

\newtheorem{rmk}[theorem]{Remark}

\newtheorem*{ex*}{Example}

\newcommand\sO{{\mathcal O}}

\newcommand\sF{{\mathcal F}}

\newcommand{\pic}[1]{{\rm Pic}^0(#1)}

\title{On the vanishing of weight one Koszul cohomology of abelian varieties}
\author{Marian Aprodu}
\address{University of Bucharest, Faculty of Mathematics and Computer Science, 14 Academiei Str., 010014 Bucharest, Romania}
\email{marian.aprodu@fmi.unibuc.ro}
\address{Simion Stoilow Institute of Mathematics of the Romanian Academy, P.O. Box 1-764, 014700 Bucharest, Romania}
\email{Marian.Aprodu@imar.ro}
\author{Luigi Lombardi}
\address{Mathematisches Institut, Universit\"at Bonn, Endenicher Allee, 60, Bonn, 53115 Germany}
\email{lombardi@math.uni-bonn.de}
\thanks{This work has started while MA was visiting the Max Planck Institute for Mathematics in Bonn (MPIM). MA thanks the MPIM for hospitality. 
MA was partly supported by the CNCS-UEFISCDI grant PN-II-ID-PCE-2012-4-0156. LL thanks the 
``Simion Stoilow'' Institute of Mathematics in Bucharest for hospitality during the preparation of this work. LL was  supported by the SFB/TR45 
``Periods, moduli spaces, and arithmetic of algebraic varieties'' of the DFG. The authors are grateful to G. Pareschi and M. Popa for having made pertinent remarks on an early version of the manuscript.}
\begin{document}
\maketitle
%\section{}
%\subsection{}

\begin{abstract}
In this Note we prove the vanishing of (twisted) Koszul cohomology groups $K_{p,1}$ of an abelian variety with values in powers of an ample line bundle. It complements the work of G. Pareschi on the property $(N_p)$ \cite{Pa}.
\end{abstract}

Let $L$ be an ample line bundle on a complex abelian variety $X$. The powers of $L$ have very interesting geometric properties. 
It is classically known that $L^2$ is globally generated, $L^3$ is very ample, and the image of $X$ in $\mathbb PH^0(X,L^3)^*$ is projectively 
normal \cite{BL}. For the fourth power of $L$, Kempf proved that the ideal of $X$ is generated by quadrics \cite{BL}. 
The syzygies of $X$ via the embeddings given by the powers $L^a$ acquire some regular behavior when $a$ grows. 
Specifically, a remarkable conjecture of R. Lazarsfeld, solved by G. Pareschi \cite{Pa}, states that for any $a\ge p+3$, $L^a$ 
has property $(N_p)$. This means that the image of $X$ is projectively normal, 
the ideal is generated by quadrics, and all the syzygies up to the $p$--th step are linear.  
(Syzygies of abelian varieties have also been studied further in \cite{PP2} and \cite{LPP}.)
The property $(N_p)$ describes the beginning of the minimal resolution of the ideal of $X$ \cite{Green}. 
It is natural to ask what happens towards the end of the minimal resolution of the ideal of $X$ 
embedded by powers of $L$ and if a similar type of regularity can be described. 
More precisely, the question raised here is which syzygies at the end of the minimal resolution are zero. 
Since syzygies can be computed by Koszul cohomology \cite{Green}, the question reduces to verifying the vanishing of corresponding Koszul 
cohomology groups.

In this Note, we prove the vanishing of Koszul cohomology groups $K_{p,1}(X,L^a)$ for suitable $p$ and $a$, Theorem \ref{main}, 
and we remark that there is a linear function in $a$ such that number of 
the zeroes at the end of the linear strand of the Betti table of $X$ with values in $L^a$ is bounded from below by this function. 
To this end, we reduce the vanishing of Koszul cohomology to the surjectivity of some suitable multiplication maps, Lemma \ref{vanishing}. 
Cokernels of multiplication maps can be expressed as twisted Koszul cohomology groups \cite{Green}, hence a  good control of twisted Koszul 
cohomology yields the surjectivity of multiplication maps. Here, we apply a reversed principle.  
To verify the surjectivity of multiplication maps we use the theory developed by G. Pareschi and M. Popa in 
\cite{Pa}, \cite{PP1}, \cite{PP2}, \cite{PP3}, and hence we eventually relate it to the vanishing of syzygies in the linear strand. 
Note that one of the main ideas in Pareschi's proof of Lazarsfeld's conjecture was also a reduction to the surjectivity of some multiplication 
maps \cite{Pa}. Another result of our Note is Theorem \ref{main3} where we obtain the vanishing of 
twisted Koszul cohomology groups  $K_{p,1}(X,B;L)$ where $L$ satisfies some positivity assumptions and $B$ is a line bundle such 
that $L-B$ is ample. As a consequence we prove in Corollary \ref{asymptotic}  an asymptotic vanishing result for line bundles of type $dL+P$ 
where $L$ is ample and globally generated and $P$ is 
arbitrary, in the spirit of \cite{EL2}.
Throughout the paper, we work over the field of complex numbers.

\section{Preliminaries}

\subsection{Regularity on abelian varieties}
Let $(X,H)$ be a polarized abelian variety and $\sF$ a coherent sheaf on $X$. We define the \emph{non-vanishing loci associated to} $\sF$ as
the algebraic closed subsets:
$$V^i(\sF):=\big\{\alpha \in \pic{X} \, | \, H^i(X,\sF \otimes \alpha)\neq 0 \big\}.$$ 
We say that $\sF$ is $M$\emph{-regular} if ${\rm codim}_{\pic{A}}\,V^i(\sF)>i$ for all $i>0$. 
Moreover we say that $\sF$ satisfies $I.T.0$ if $V^i(\sF)=\emptyset$ for all $i>0$. 
Note that a line bundle on an abelian variety is ample if and only if it is $M$-regular, and hence $I.T.0$ \cite{PP}. 
Recall also that the tensor product of an $I.T.0$ bundle with a nef line bundle is $I.T.0$ 
(\cite[Theorem B]{PPgv} and \cite[Proposition 3.1]{PP3}), 
and the tensor product of two $I.T.0$ bundles remains $I.T.0$ \cite[Proposition 2.9]{PP1}. 
Pareschi and Popa proved furthermore that the tensor product of two $M$-regular bundles is $M$-regular \cite[Theorem 3.9]{PP}.
% Finally, following \cite{PP1}, we say that $\sF$ is $l$-$H$-regular if $\sF((l-1)H)$ is $M$-regular.

\subsection{Surjectivity of multiplication maps}
We recall some results on the surjectivity of multiplication maps of global sections of vector bundles on abelian varieties. These results will be 
our main tools to check property $(M_q)$ on abelian varieties. 

Given an abelian variety $X$, we denote by $m:X\times X\rightarrow X$ the multiplication map and by $p_1,p_2:X\times X\rightarrow X$ the projections 
onto the first and second factor, respectively. For a line bundle $L$ and a vector bundle $E$ on $X$ such that $L\otimes E$ satisfies $I.T.0$, 
we define the \emph{skew Pontrjagin product} of $L$ with $E$ as the vector bundle 
$$L\, \widehat{*}\, E \, := \, p_{1*}\big(m^*L\otimes p_2^*E\big).$$
This bundle is one of the main objects used by Pareschi in \cite{Pa} to prove the $(N_p)$ property of line bundles on abelian varieties.

\begin{theorem}[Kempf, Pareschi, Pareschi--Popa]
\label{surjthm}
  Let $(X,H)$ be a polarized abelian variety and $E$ and $F$ be vector bundles on $X$. Moreover let $L$ be a line bundle. 
  \begin{itemize}
\item[(i)] (\cite[Theorem 7.34]{PP3}) If 
$E(-2H)$ and $F(-2H)$ are $M$-regular sheaves, then the multiplication map  $H^0(X,E)\otimes H^0(X,F)\rightarrow H^0(X,E\otimes F)$ is surjective.  
\\

\item[(ii)] (\cite[Theorem 3.1]{Pa}) 
   If both $L\otimes E$ and 
   $(L\, \widehat{*}\, E)(-H)$ satisfy $I.T.0$, then the multiplication map
   $H^0(X,L)\otimes H^0(X,E)\rightarrow H^0(X,L\otimes E)$ is surjective.
\\ 
 
\item[(iii)]  (Kempf, see \cite[\S3, p. 660]{Pa}) 
   If for some integer $a>1$ the bundle $E\otimes L^a$ satisfies $I.T.0$ and 
   $H^i \big(X,(L^a \, \widehat{*} \, (L^{-1}\otimes \alpha)) \otimes E \big)=0$ for all $i>0$ and $\alpha \in \pic{X}$, then the multiplication map 
   $H^0(X,L^a)\otimes H^0(X,E)\rightarrow H^0(X,L^a\otimes E)$ is surjective.
 \end{itemize}
 \end{theorem}
 
 \subsection{Koszul cohomology and property $(M_q)$}
 This time we denote by $X$ an arbitrary smooth projective variety. Let $L$ be a globally generated line bundle on $X$ and denote by 
 $$M_L \, := \, {\rm ker}\big(H^0(X,L)\otimes \sO_X\stackrel{ev}{\longrightarrow} L\big)$$ the kernel bundle associated to $L$ defined as kernel 
 of the evaluation map of $L$. Note that the 
 rank of $M_L$ is $r:=h^0(X,L)-1$.
 Moreover let $B$ be an arbitrary line bundle on $X$. For integers 
 $p\geq 0$ and $q\geq 0$ we define  \emph{twisted Koszul cohomology} groups as (R. Lazarsfeld \cite{L}, \emph{cf}. also \cite[Remark 2.7]{AN}):
 $$K_{p,q}(X,B;L) \; := \; {\rm ker}\big( H^1(X,\wedge^{p+1}M_L\otimes L^{q-1}\otimes B)\stackrel{\varphi_{p,q,B,L}}{\longrightarrow} \wedge^{p+1}H^0(X,L)\otimes H^1(X,L^{q-1}\otimes B)\big)$$
 where the map $\varphi_{p,q,B,L}$ is induced by the following short exact sequence passing to cohomology
 $$0\longrightarrow \wedge^{p+1}M_L\otimes L^{q-1}\otimes B \longrightarrow \wedge^{p+1} H^0(X,L)\otimes L^{q-1}\otimes B \longrightarrow \wedge^pM_L \otimes L^q \otimes B\longrightarrow 0.$$
The word ``twisted'' is related to the presence of $B$ in the picture. If  $B\simeq\mathcal \sO_X$, then we simply speak about {\em Koszul cohomology}.  As a matter of notation, if $B$ is trivial, we  set $K_{p,q}(X,L):=K_{p,q}(X,\mathcal O_X;L)$. If $L$ is very ample and the image of $X$ is projectively normal, then Koszul cohomology groups compute the syzygies of $X$ \cite{Green}. The table formed with the dimensions of the spaces $K_{p,q}(X,L)$ varying $p$ and $q$ is called the \emph{Betti table} of $X$ with values in $L$.
 
The Koszul cohomology groups $K_{p,1}(X,L)$ play a particularly important role. From the syzygetic view-point, $K_{1,1}(X,L)$ computes the number of quadrics in the ideal of the image of $X$, $K_{2,1}(X,L)$ computes the linear relations among those quadrics, $K_{3,1}(X,L)$ computes the linear relations among the linear relations among those quadrics etc. For this reason, the raw in the Betti table corresponding the $K_{p,1}$'s is called the \emph{linear strand}. 

An important problem in syzygy theory is to know when we stop having linear syzygies. This is controlled by the following definition.
We say that {\em $L$ satisfies property $(M_q)$} if 
 $$K_{p,1}(X,L) \; = \; 0 \quad \mbox{ for \; all } \quad p\geq r-q.$$
 This property $(M_q)$ was introduced by M. Green and R. Lazarsfeld in a slightly different form \cite{GL}.  
One manifestation of this property is in connection with the gonality of curves. 
It has been conjectured by Green and Lazarsfeld, and proved recently by L. Ein and R. Lazarsfeld \cite{EL2}, that on a 
smooth curve $X$ of gonality $d$ we have $K_{p,1}(X,L)=0$ for any line bundle $L$ of sufficiently large degree and 
any $p\ge h^0(X,L)-d$. In other words, the property $(M_{d+1})$ holds for $L$. Hence the gonality of a curve can be read 
off from the tail of linear strand of the Betti table. We remark that 
for line bundles of sufficiently large degree on a curve there is no difference 
between the original definition of $(M_q)$ and the definition we work with.

On varieties of arbitrary dimension a geometric manifestation of (the failure of) the property $(M_q)$ is through syzygy 
schemes \cite{Green}, \cite{Ehbauer}, \cite{AN}. A {\em syzygy scheme} is a scheme containing $X$ whose equations are the quadrics 
involved in a given non-zero class in $K_{p,1}(X,L)$. The intersection of all the syzygy schemes associated to classes in $K_{p,1}(X,L)$ is 
called the {\em syzygy scheme of weight $p$}. If $p$ is very close to $r$, 
the syzygy schemes of weight $p$ can be classified, and turn out to be 
varieties of small degree \cite[Theorem (3.c.1)]{Green}, \cite[Theorem 7.1]{Ehbauer}. Hence one reason for the non-vanishing of $K_{p,1}$ 
for large $p$ is that the original variety lies on varieties with special geometry. The classification of syzygy schemes for smaller $p$ is a 
very interesting  open  problem.

\section{Twisted Koszul cohomology groups of weight one}

In this section we study the vanishing of Koszul cohomology groups of type
$K_{p,1}(X,B;L)$ where $X$ is an abelian variety of dimension $n$, $L$ is a sufficiently ample line bundle with $r:=h^0(X,L)-1$, 
and $B$ is a line bundle such that $L-B$ is ample. 
The property $(M_q)$ for $L$ follows at once by applying our vanishings to $B\simeq\sO_X$. From the $K_{p,1}$-Theorem \cite[Theorem (3.c.1)]{Green} 
we obtain the vanishing of $K_{r-n,1}(X,L)$. However, we can prove much more using the techniques developed by Pareschi and Popa.

\subsection{Sufficient conditions for the vanishing of $K_{p,1}$}
The following lemma gives a sufficient criterion for the vanishing 
of twisted Koszul cohomology groups of type $K_{p,1}(X,B;L)$ 
in terms of multiplication maps of global sections of special vector bundles. 
We remark that it works for any complex smooth projective variety and not just for abelian varieties.
\begin{lemma}
\label{thm-main}
\label{vanishing}
Let $X$ be a complex smooth projective variety of dimension $n\ge 2$ 
and $L$ be an ample and globally generated line bundle. Set $r:=h^0(X,L)-1$ and let $1\le p\le r-n$ be an integer. Moreover let  
$B$ be another line bundle on $X$ such that $L-B$ is ample and denote by $\omega_X$ the canonical bundle of $X$. 
If the multiplication maps 
$$\mu_k \, : \, H^0(X,L)\otimes H^0\big(X,M_L^{\otimes k} \otimes \omega_X \otimes L^{n-1}\otimes B^{-1} \big)\longrightarrow 
H^0 \big(X,M_L^{\otimes k} \otimes \omega_X \otimes L^{n}\otimes B^{-1} \big)$$
are surjective for all $k=0,\ldots,r-n-p$, then $H^1(X,\wedge^{p+1}M_L\otimes B)=0$ and, in particular, $K_{p,1}(X,B;L)=0$.
\end{lemma}

\proof
The vanishing of $H^1 \big(X,\wedge^{p+1}M_L\otimes B \big)$ is equivalent, via Serre duality and the isomorphism $\wedge^{p+1}M_L^{\vee} \simeq 
\wedge^{r-p-1}M_L\otimes L$, to the vanishing
\begin{equation}\label{eq22}
H^{n-1}\big(X,\wedge^{r-p-1}M_L\otimes \omega_X\otimes L\otimes B^{-1} \big) \, = \, 0.
\end{equation}
We claim there is an isomorphism 
\[
H^{n-1} \big(X,\wedge^{r-p-1}M_L\otimes \omega_X\otimes L\otimes B^{-1} \big)\simeq H^1 
\big(X,\wedge^{r-p-n+1}M_L\otimes \omega_X\otimes L^{n-1}\otimes B^{-1} \big).
\]
This is clear for $n=2$, while for $n>2$ it follows from Kodaira vanishing applied to the following exact sequences:
{
\[
0\to \wedge^{r-p-1}M_L\otimes \omega_X\otimes L\otimes B^{-1} \to \wedge^{r-p-1}H^0(X,L)\otimes \omega_X\otimes L\otimes B^{-1} \to
\]
\[
\to
 \wedge^{r-p-2}M_L\otimes \omega_X\otimes L^2\otimes B^{-1} \to 0
\]

\[
0\to \wedge^{r-p-2}M_L\otimes \omega_X\otimes L^2\otimes B^{-1} \to\wedge^{r-p-2}H^0(X,L)\otimes \omega_X\otimes L^2\otimes B^{-1} \to 
\]
\[
\to\wedge^{r-p-3}M_L\otimes \omega_X\otimes L^2\otimes B^{-1} \to0
\]

\[
\ldots
\]

\[
0\to  \wedge^{r-p-n+2}M_L\otimes \omega_X\otimes L^{n-2}\otimes B^{-1}\to \wedge^{r-p-n+2}H^0(X,L)\otimes \omega_X\otimes L^{n-2} \otimes B^{-1}\to 
\]
\[
\to\wedge^{r-p-n+1}M_L\otimes \omega_X\otimes L^{n-1}\otimes B^{-1} \to 0.
\]
}
Since in characteristic zero exterior powers are direct summands of tensor products, in order to prove \eqref{eq22} it suffices to show the 
vanishing
\begin{equation}
\label{h1=0}
H^1 \big(X,M_L^{\otimes (r-p-n+1)}\otimes \omega_X\otimes L^{n-1}\otimes B^{-1} \big) \, = \, 0.
\end{equation}
To this end, we note that the surjectivity of $\mu_0$, Kodaira vanishing,
and the following exact sequence
\[
 0\to M_L\otimes  \omega_X\otimes L^{n-1}\otimes B^{-1} \to H^0(X,L)\otimes \omega_X\otimes L^{n-1}\otimes B^{-1}\to  L^n\otimes \omega_X
 \otimes B^{-1}\to 0
\]
immediately yield
\begin{equation}\label{M1}
H^1 \big(X,M_L\otimes \omega_X\otimes L^{n-1}\otimes B^{-1} \big) \, = \, 0.
\end{equation}
By tensorizing the previous sequence by $M_L$, we see that the vanishing in \eqref{M1}, together with the surjectivity of $\mu_1$, yields the vanishing
$$H^1 \big(X,M_L^{\otimes 2}\otimes \omega_X\otimes L^{n-1}\otimes B^{-1} \big) \, = \, 0.$$ 
Therefore we get (\ref{h1=0}) by proceeding in this way $r-n-p$ times.

\endproof

\begin{rmk}
The previous theorem holds in particular if $B^{-1}$ is nef, but more generally if one has the following vanishings
\[
H^1(X,\omega_X\otimes L^{n-1}\otimes B^{-1}) \, = \, H^1(X,\omega_X \otimes L^{n-1}\otimes B^{-1}) \, = \, 
H^2(X,\omega_X \otimes L^{n-2}\otimes B^{-1}) \, = \, \ldots \, = \, 
 \]
 \[
 =\, H^{n-3}(X,\omega_X\otimes L^2\otimes B^{-1}) \, = \, H^{n-2}(X,\omega_X\otimes L^2\otimes B^{-1}) \, = \, H^{n-2}(X,\omega_X\otimes L\otimes B^{-1}) \, = 
 \]
 \[= \, H^{n-1}(X,\omega_X \otimes L\otimes B^{-1}) \, = \, 0
\]
in place of the ampleness of $L-B$.
\end{rmk}

\begin{rmk}
 The sufficient conditions in Lemma \ref{vanishing} are not necessary. For instance, let $X$ be a $K3$ surface
 and let $L$ be a very ample line bundle on $X$ such that the hyperplane section is a hyperelliptic curve. For $p=r-2$ and $B\simeq\mathcal O_X$
 we obtain $K_{r-2,1}(X,L)=0$ since $X$ is not of minimal degree \cite[Theorem (3.c.1)]{Green}. On the 
 other hand, the multiplication map
 \[
  H^0(X,L)\otimes H^0(X,L)\to H^0(X,L^2)
 \]
 is not surjective, as the hyperplane section is hyperelliptic. Indeed, the cokernel of this map is isomorphic
 to $K_{0,2}(X,L)$ which is furthermore isomorphic to $K_{0,2}(C,K_C)$ by the hyperplane section Theorem
 on Koszul cohomology, \cite[Theorem (3.b.7)]{Green}.
\end{rmk}

\subsection{Powers of a line bundle}
In case of powers of a fixed line bundle we can use Theorem \ref{surjthm} (iii) 
to establish the surjectivity of multiplication maps appearing in Lemma \ref{vanishing}. We prove:

\begin{theorem}\label{main}
 Let $L$ be an ample line bundle on an abelian variety $X$ of dimension $n\geq 3$.
 Let $a\geq 2$ be an integer and set $r_a:=h^0(X,L^a)-1$. 
 Moreover let $B$ be a line bundle on $X$ such that $bL-B$ is ample for some integer $b\geq 1$.
 Then for $p\leq r_a-n$ we have 
  $$K_{p,1}(X,B;L^a) \, = \, 0$$ 
  
 \noindent provided that 
\begin{equation}\label{d1}
p \, \geq \, r_a\, - \, a\, (n \, - \, 1) \, + \, b\,\Big(1 \, - \, \frac{1}{a}\Big) \quad \mbox{ and }\quad a \, \geq \, b.
\end{equation}

\end{theorem}
 \begin{proof}

First of all we notice that $L^a$ is globally generated for $a\geq 2$.
According to Lemma \ref{thm-main} we need to prove the surjectivity of multiplication maps 
$$\mu_k \, : \, H^0(X,L^a)\otimes H^0\big(X,M_{L^a}^{\otimes k} \otimes L^{a(n-1)}\otimes B^{-1}\big)\,
\longrightarrow \, H^0\big(X,M_{L^a}^{\otimes k} \otimes L^{an}\otimes B^{-1}\big)$$ for 
$k=0,\ldots , r_a-n-p$ (note that $aL-B$ is ample as $a\geq b$). To this end we use Theorem \ref{surjthm} (iii) and we start with the case $k=0$. 
Then the map 
$$\mu_0 \, : \, H^0(X,L^a)\otimes H^0(X,L^{a(n-1)}\otimes B^{-1}) \, \longrightarrow \,
H^0(X,L^{an}\otimes B^{-1})$$ is surjective as soon as the line bundle
$L^{an}\otimes B^{-1}$ satisfies $I.T.0$ (which is true as $a\geq b> \frac{b}{n}$) and 
\begin{equation}\label{van}
H^i\big(X,L^a\, \widehat{*} \, (L^{-1}\otimes \alpha) \otimes L^{a(n-1)}\otimes B^{-1}\big) \, = \, 0 \quad \mbox{ for \, all }\quad \alpha 
\, \in \, \pic{X} \, \mbox{ and } \, i \, > \, 0.
\end{equation}
To prove these vanishing we pull-back the involved bundles via the multiplication by $(a-1)$ map which we denote by $(a-1)_X:X\rightarrow X$.
Then by \cite[Proposition 3.6]{Pa} the pull-back 
$(a-1)_X^*\big(L^a\, \widehat{*} \, (L^{-1}\otimes \alpha))$ is isomorphic to a trivial bundle times a line bundle algebraically equivalent to 
$\big(-a(a-1)\big)L$. On the other hand, the pull-back $(a-1)_X^* \big(L^{a(n-1)}\otimes B^{-1} \big)$ is isomorphic to a line bundle 
algebraically equivalent to 
$a(a-1)^2(n-1)L-(a-1)^2B$. Hence in order to prove \eqref{van}, by applying Kodaira vanishing, it suffices to prove that the line bundle 
$$\Big(a\, (a\, - \, 1)^2\, (n\, - \, 1) \, - \, a (a-1)\Big) \,L \, - \, (a\, - \, 1)^2 \, B$$ is ample. But this is the case as soon as
\begin{equation}\label{eq21}
a\, (a \, - \, 1)\, (n\, - \, 1) \, - \, a\, \geq \, b\, (a\, - \, 1)
\end{equation}
which holds true by the second inequality of \eqref{d1}.

We assume now $k>0$. In order to prove the surjectivity of $\mu_k$ for $k=1,\ldots, r_a-n-p$, by Theorem \ref{surjthm} (iii)
it suffices to prove the vanishings:
\begin{equation}\label{F0}
 H^i \big(X,M_{L^a}^{\otimes k}\otimes L^{an}\otimes B^{-1}\otimes \beta \big) \, = \, 0\quad \mbox{ for \, all }\quad i \, > \, 0 
 \quad \mbox{ and }\quad \beta \, \in \, \pic{X}
\end{equation}
\begin{equation}\label{F}
H^i\big(X,(L^a \, \widehat{*} \, (L^{-1}\otimes \alpha))\otimes M_{L^a}^{\otimes k}\otimes L^{a(n-1)} \otimes B^{-1}\big) \, = \, 0\quad
\mbox{ for \, all } \quad i \, > \, 0\quad 
\mbox{ and } \quad \alpha \, \in \, \pic{X}.
\end{equation}
Set $F_{\alpha}(a):=L^a \, \widehat{*} \, (L^{-1}\otimes \alpha)$ and consider the short exact sequences:
\[
0\rightarrow M_{L^a}^{\otimes k}\otimes L^{an}\otimes B^{-1}\otimes \beta \rightarrow H^0(X,L^a)\otimes M_{L^a}^{\otimes(k-1)}\otimes L^{an}\otimes B^{-1}\otimes \beta 
\rightarrow \]\[M_{L^a}^{\otimes(k-1)}\otimes L^{a(n+1)}\otimes B^{-1}\otimes \beta \rightarrow 0
\]

\[
0\rightarrow M_{L_a}^{\otimes k}\otimes F_{\alpha}(a)\otimes L^{a(n-1)}\otimes B^{-1} \rightarrow H^0(X,L^a)\otimes M_{L^a}^{\otimes(k-1)}\otimes F_{\alpha}(a)
\otimes L^{a(n-1)}\otimes B^{-1}
\] 
\[
\rightarrow M_{L^a}^{\otimes(k-1)}\otimes F_{\alpha}(a) \otimes L^{an}\otimes B^{-1}\rightarrow 0.
\]

We denote by $\mu(E,E'):H^0(X,E)\otimes H^0(X,E')\rightarrow H^0(X,E\otimes E')$ the multiplication map of global sections of 
two vector bundles on $X$.
We see then that \eqref{F0} holds as soon as the multiplication maps 
\begin{equation}\label{eq23}
\mu\big(L^a,M_{L^a}^{\otimes(k-1)}\otimes L^{an}\otimes B^{-1}\otimes \beta \big)
\end{equation}
are surjective for all 
$\beta$ and both the sheaves $M_{L^a}^{\otimes(k-1)}\otimes L^{an}\otimes B^{-1}$ 
and $M_{L^a}^{\otimes(k-1)}\otimes L^{a(n+1)}\otimes B^{-1}$ satisfy $I.T.0$. On the other hand, 
\eqref{F} holds if 
the multiplication maps 
\begin{equation}\label{eq24}
\mu \big(L^a, M_{L^a}^{\otimes(k-1)}\otimes F_{\alpha}(a)\otimes L^{a(n-1)} \otimes B^{-1}\big)
\end{equation}
are surjective 
for all $\alpha\in \pic{X}$ 
and both the sheaves $M_{L^a}^{\otimes(k-1)}\otimes F_{\alpha}(a) \otimes L^{a(n-1)}\otimes B^{-1}$ and 
$M_{L^a}^{\otimes(k-1)}\otimes F_{\alpha}(a) \otimes L^{an}\otimes B^{-1} $
satisfy $I.T.0$ for all $\alpha \in \pic{A}$. 
Since the tensor product of an $I.T.0$ locally free sheaf with an ample line bundle again satisfies $I.T.0$,  
by applying Theorem \ref{surjthm} (iii) to the maps \eqref{eq23} and \eqref{eq24}, we deduce that \eqref{F0} and \eqref{F} hold if 
$$F_{\beta}(a) \otimes M_{L^a}^{\otimes(k-1)}\otimes L^{a(n-1)}\otimes B^{-1}\quad \mbox{ satisfy }\quad I.T.0 \quad \mbox{ for \, all }
\quad \beta \in \pic{X} \quad \mbox{ and}$$

$$F_{\alpha}(a)\otimes F_{\alpha'}(a)\otimes M_{L^a}^{\otimes(k-1)}\otimes L^{a(n-1)} \otimes B^{-1}
\quad \mbox{ satisfy }\quad I.T.0 \quad \mbox{ for \, all }\quad \alpha,\alpha' \in \pic{X}\quad \mbox{ and}$$

 $$M_{L^a}^{\otimes(k-1)}\otimes L^{a(n-1)}\otimes B^{-1} 
 \quad \mbox{ satisfies }\quad I.T.0.$$
Proceeding in this way $k-1$ times, we see that all we need to show is that for all $1\leq k'\leq k+1$ the sheaves 
 \begin{equation}\label{F2}
 \bigotimes _{i=1}^{k'} F_{\alpha_i}(a) \otimes L^{a(n-1)}\otimes B^{-1}\quad \mbox{ satisfy }\quad I.T.0 \quad \mbox{ for any choice of }\quad 
 \alpha_1,\ldots ,\alpha_{k'}\in \pic{X} \quad \mbox{ and}
 \end{equation}
 $$L^{a(n-1)}\otimes B^{-1}\quad \mbox{ satisfies } \quad I.T.0.$$
 
 Since $a(n-1)\geq b$ we have that $L^{a(n-1)}\otimes B^{-1}$ is ample and hence $I.T.0$. 
 Now we determine the range of $a$ for which the bundles in \eqref{F2} are $I.T.0$.
 We notice that a bundle satisfies $I.T.0$ if its pull-back under $(a-1)_X$ does so.
 By \cite[Proposition 3.6]{Pa} the bundle $(a-1)_X^*F_{\alpha}(a)$ is isomorphic to a trivial bundle times a line bundle 
 algebraically equivalent to 
 $\big(-a(a-1)\big)L$. Therefore $(a-1)_X^* \big( \bigotimes_{i=1}^{k'}F_{\alpha_i}(a) \big)$ 
 is isomorphic to a trivial bundle times a line bundle algebraically 
 equivalent to $-ak'(a-1)L$. On the other hand $(a-1)_X^*\big(L^{a(n-1)}\otimes B^{-1}\big)$ is isomorphic 
 to a trivial bundle times a line bundle algebraically equivalent to 
 $$\big((a\, - \, 1)^2 \, a \, (n\, - \, 1)\big) \, L \, - \, (a\, - \, 1)^2 \, B.$$ 
 In conclusion $$(a-1)_X^* \Big(\bigotimes _{i=1}^{k'} F_{\alpha_i}(a) \otimes L^{a(n-1)}\otimes B^{-1}\Big)$$ is 
 isomorphic to a trivial bundle times a line bundle algebraically equivalent to 
 \begin{equation}\label{eq2}
 \Big((a \, - \, 1)^2 \, a\, (n \, -\, 1) \, - \, a \, k'\, (a \, - \, 1) \Big) \, L \, - \,  (a\, - \, 1)^2\, B,
 \end{equation}
 which in turn satisfies $I.T.0$ as soon as
 $$(a \, - \, 1)^2 \, a\, (n\, -\, 1)-a \, k'\, (a\, -\, 1) \, \geq \, b \, (a\, - \, 1)^2.$$
 Hence $\mu_{k}$ is surjective for all $k=1,\ldots ,r_a-n-p$ as soon as 
 $$(a \, -\, 1)\, a\, (n\, -\, 1) \, - \, a \, (r_a \,- \, n\, -\, p\, +\, 1)\, \geq \, b\, (a\, -\, 1),$$
 but this is equivalent to asking that 
 \begin{equation}\label{quad}
 (n-1)\,a^2\, - \, (r_a-p+b)\,a\, + \, b\, \geq \, 0,
 \end{equation}
which in turn is equivalent to the first inequality of \eqref{d1}.
 \end{proof}

  \begin{rmk}\label{n=2}
  Theorem \ref{main} continues to hold if one assumes 
  $n = 2$, $a\geq 3$, and $a>b$ in place of $n\geq 3$, $a\geq 2$, and $a\ge b$. The only part where this affects
  the proof is in \eqref{eq21}. 
  \end{rmk}
  
  \begin{rmk}
From  Theorem \ref{main}, for $a=2$ and $n\geq 3$ 
we obtain the vanishing $$K_{p,1}(X,L^2) \, = \, 0\quad \mbox{ for }\quad p\geq r_2-2n+3.$$ 
Note that in \cite{Pa} property $(N_p)$ was established from the cubes of a line bundle on. On the other hand, in \cite{PP2}, property 
$(N_p)$ was proved 
for the square of a line bundle with no base divisors. Here we see that in general $L^2$ satisfies $(M_{2n-3})$ without any further hypotheses. 
% It might be possible to obtain a better vanishing if we assume that $L$ has no base divisor, cf.~\cite{PP2}. 
\end{rmk}
 Theorem \ref{main} immediately gives information on property $(M_q)$ for powers of an ample line bundle on an abelian variety, for a fixed $q$.
 
 \begin{cor}\label{Mqprop} 
  Let $X$ be an abelian variety of dimension $n\geq 2$ and let $L$ be an ample line bundle. 
  For any fixed $q\geq n$, the line bundle $L^{a}$ satisfies condition $(M_q)$ for all  
  \begin{equation}\label{a}
  a\; \geq \; \frac{q+1+\sqrt{(q+1)^2-4(n-1)}}{2(n-1)}.
  \end{equation}
 \end{cor}
 
\begin{proof}
Assume first $n\geq 3$ so that, without loss of generality, we can suppose $a\geq 2$ from \eqref{a} (in particular $L^a$ is globally generated).
In the proof of Theorem \ref{main} we set $B\simeq\sO_X$, $b=1$, and $q:=r_a-p$.  
Then we notice that the discriminant of (\ref{quad}) 
in terms of $q$ is $(q+1)^2-4(n-1)$ which is non-negative for $q\geq n$. Therefore \eqref{quad} admits a positive solution which is 
given by the right-hand side of \eqref{a}.
If $n=2$, then we can suppose $a\geq 3$ from \eqref{a}. Then we apply Remark \ref{n=2} and argue as in the previous case. 
% Also note that the right-hand side of \eqref{a} is greater than $1$ so that  the second equation in \eqref{d1} does not play any role. 
 \end{proof}
 The following corollary gives a geometric interpretation of the non-vanishing of Koszul cohomology of weight one. 
 
 \begin{cor}
  Let $L$ be an ample line bundle on an abelian variety $X$ of dimension $n\geq 3$ and let $a\geq 2$ be an integer (or choose $n = 2$ and $a\geq 3$). 
  If for some $q\geq 0$ we have $K_{r_a-n-q,1}(X,L^a)\neq 0$ where $r_a:=h^0(X,L^a)-1$, then 
%  $$n\, < \, \frac{2 \, a^2 \,+ \, a\, q \, -\, 1\, +\, \sqrt{(a\,q\,-\,1)^2\,+\,4\,a\,(a\,q\,-\,1)\,+\,4\, a^2}}{2\, a \, (a \, - \, 1)}$$ 
 $$ n \; < \; \frac{a+1}{a}+\frac{q+1}{a-1}.$$
 \end{cor}
 
\begin{proof}
 In Theorem \ref{main} (or Remark \ref{n=2}) we set $B\simeq \sO_X$, $b=1$ and $p=r_a-n-q$.
 Therefore by the same either $p>r_a-n$ (which is excluded by our choice of $p$), or 
 $$a \; < \; \frac{n+q+1+\sqrt{(n+q+1)^2-4(n-1)}}{2(n-1)}$$ which translates into the stated inequality for $n$.
 \end{proof}

\subsection{The general case}

In this subsection we prove the vanishing of Koszul groups of type $K_{p,1}(X,B;L)$ on a polarized abelian variety under some positivity conditions 
on the line bundles $L$ and $B$. Although the following result applies to a wider class of line bundles, 
we remark that for the case of positive powers of a fixed ample line bundle our previous Theorem \ref{main} yields a stronger 
vanishing.
 
\begin{theorem}\label{main3}
 Let $L$ be a globally generated line bundle on a polarized abelian variety $(X,H)$ of dimension $n\geq 3$ such that $L-2H$ is ample.
 Set $r:=h^0(X,L)-1$ and moreover let $B$ be a line bundle on $X$ such that $L-B$ is ample. 
 Then for $p\leq r-n$ we have
 $$K_{p,1}(X,B;L)=0$$  
 if either $p>r-2n+2$, or $p=r-2n+2$ and $-B$ is nef. 
\end{theorem}

\begin{proof}
We apply Lemma \ref{vanishing} so that we only need to verify the surjectivity of the multiplication maps
$$\mu_k \, : \, H^0(X,L)\otimes H^0(X, M_L^{\otimes k}\otimes L^{n-1}\otimes B^{-1}) \, \longrightarrow 
\, H^0(X,M_L^{\otimes k} \otimes L^{n} \otimes B^{-1}) $$ for all $k= 0,\ldots, r-n-p.$
We start with the case $k=0$ and will show that 
$$\mu_0 \, : \, H^0(X,L)\otimes H^0(X,L^{n-1}\otimes B^{-1}) \, \longrightarrow \, H^0(X,L^n\otimes B^{-1})$$ is surjective. 
Since $n\ge 3$, it follows that both $L(-2H)$ and $(L^{n-1}\otimes B^{-1})(-2H)$ are ample. By applying Theorem \ref{surjthm} (i) we obtain 
then the surjectivity of $\mu_0$.

We prove now the surjectivity of $\mu_{k+1}$ for $k = 0,\ldots , r-n-p-1$ step by step. 
By Theorem \ref{surjthm} (i) this follows if one proves that the bundle $M_L^{\otimes(k+1)}\otimes L^{n-1}(-2H)\otimes B^{-1}$ is $M$-regular. We show actually that it satisfies $I.T.0$.
By looking at the sequence 
\begin{equation}\label{exact2}
0\longrightarrow M_L^{\otimes(k+1)}\otimes L^{n-1}(-2H)\otimes B^{-1} \otimes \alpha \longrightarrow H^0(X,L)\otimes M_L^{\otimes k}\otimes L^{n-1}(-2H) 
\otimes B^{-1}\otimes \alpha 
\end{equation}
\[
\longrightarrow M_L^{\otimes k} \otimes L^{n}(-2H) \otimes B^{-1} \otimes \alpha \longrightarrow 0
\]
we infer that $M_L^{\otimes(k+1)}\otimes L^{n-1}(-2H)\otimes B^{-1}$ satisfies $I.T.0$ if the following two conditions are true (we still use the 
notation $\mu(E,E'):H^0(X,E)\otimes H^0(X,E')\rightarrow H^0(X,E\otimes E')$ to denote the multiplication map of global sections 
of two vector bundles):
\begin{equation}\label{Mreg}
M_L^{\otimes k}\otimes L^{n-1}(-2H)\otimes B^{-1} \quad \mbox{ satisfies }\quad I.T.0
\end{equation}
\begin{equation}\label{mmsur}
\mu \big(L, M_L^{\otimes k}\otimes L^{n-1}(-2H)\otimes B^{-1} \otimes \alpha \big)\quad \mbox{ is \,surjective \,for \,any }\quad  \alpha \in \pic{X}
\end{equation}
(one also needs that $M_L^{\otimes k} \otimes L^{n}(-2H) \otimes B^{-1}$ satisfies $I.T.0$, but this follows from \eqref{Mreg}).
On the other hand, by means of Theorem \ref{surjthm} (i), it is easy to see that both \eqref{Mreg} and \eqref{mmsur} hold true as soon as 
$$M_L^{\otimes k}\otimes L^{n-1}(-4H)\otimes B^{-1}  \quad \mbox{ satisfies }\quad I.T.0.$$
Proceeding this way $k$ times, we are reduced to check that $$M_L\otimes L^{n-1}\big(-2(k+1)H \big)\otimes B^{-1} \quad \mbox{ satisfies }\quad I.T.0.$$
In turn this follows from the ampleness of 
$L^{n-1}\big(-2(k+2)H\big)\otimes B^{-1}$  which is ensured by the assumption $p> r-2n+2$ (which implies that $n-1> k+2$), and the assumption on 
the nefness of $B^{-1}$ if $p=r-2n+2$.
\end{proof}

\begin{rmk}
The result remains valid for $n=2$ if in addition we assume that $B^{-1}$ is nef. Indeed, our numerical conditions force 
$p=r-2$, and hence we have only one multiplication map involved, namely $\mu_0$.
\end{rmk}

 \begin{rmk}
 The case of elliptic curves is much more specific. Let $X$ be an elliptic curve,  $L$ be a line bundle of degree $r+1$ with $r\ge 3$ and $B$ be an arbitrary a line bundle on $X$. We can identify all the cases when $K_{r-1,1}(X,B;L)=0$.
 
 From Green's duality theorem, \cite[Theorem (2.c.6)]{Green} (note that the duality does not apply in general for abelian varieties of dimension $\geq 2$), we have an isomorphism
 \[
 K_{r-1,1}(X,B;L)^*\simeq K_{0,1}(X,B^{-1};L),
 \]
 and the latter is isomorphic, by definition, to the cokernel of the multiplication map
 \begin{equation}\label{eqn:multiplication map}
 H^0(X,L)\otimes H^0(X,B^{-1})\longrightarrow H^0(X,L-B).
 \end{equation}
If $\mathrm{deg}(B)\ge 0$ and $B\not\simeq O_X$, then $H^0(X,B^{-1})=0$ and hence we obtain $K_{r-1,1}(X,B;L)\neq 0$. 
%Moreover, from the duality theorem again we have $K_{r,1}(X,B;L) = 0$. 
If $B\simeq \mathcal O_X$ then the map (\ref{eqn:multiplication map}) is obviously an isomorphism and hence $K_{r-1,1}(X,L)=0$. 
Note that the non-vanishing Theorem of Green and Lazarsfeld \cite[Appendix]{Green} for a decomposition $L\simeq L(-2x)\otimes \mathcal O_X(2x)$ with $x\in X$ implies that $K_{r-2,1}(X,L)\neq 0$. 
If $\mathrm{deg}(B)\leq -2$, since $h^1(X,L^{-1} \otimes B^{-1})=h^0(X,L\otimes B)\leq h^0(X,L)-2$ ($L$ is very ample), we may apply Green's  $H^0$ Lemma (\cite[Theorem (4.e.1)]{Green} to conclude that the multiplication map (\ref{eqn:multiplication map}) is surjective, i.e. $K_{r-1,1}(X,B;L) = 0$. If $B=\mathcal O_X(-x)$ with $x\in X$ then the multiplication map (\ref{eqn:multiplication map}) is not surjective which implies $K_{r-1,1}(X,B;L)\neq 0$. 
 \end{rmk}

It is possible to give an asymptotic version of Theorem \ref{main3} for line bundles of type $L_d:=d L+P$ where $L$ is ample and 
globally generated and $P$ arbitrary.
  
\begin{cor}
\label{asymptotic}
Let $L,P,B$ be line bundles on an abelian variety $X$ such that $L$ is ample and globally generated. Set $L_d:=dL+P$ and $r_d:=h^0(X,L_d)-1$.
Then there exists an integer $d_0=d_0(L,P,B)$ such that if $d\geq d_0$, then $K_{p,1}(X,B;L_d)=0$ for all $r_d-2n+2 < p \leq r_d-n$.
If in addition $-B$ is nef, then the statement is also true for $p=r_d-2n+2$.
\end{cor}

\begin{proof}
With the notation of Theorem \ref{main3}, 
we set $H=L$ and choose integers $d_1=d_1(L)$ such that $d_1L$ is very ample (\emph{e.g.} $d_1=3$) and $d_2=d_2(L,P)$ such that  
$(d_2-2)L+P$ is globally generated. Moreover let $d_3=d_3(L,P,B)$ be such that $d_3L+P-B$ is globally generated. 
Then for any $d\geq d_0:={\rm max}\{d_1+d_2,d_1+d_3\}$ the line bundles $L_d-2L$ and $L_d-B$ are very ample. 
Moreover $L_d$ is globally generated for all $d\geq d_0$. 
Now the corollary follows by Theorem \ref{main3}.
\end{proof}

 As a final comment, it is natural to ask how far are the results included here from being optimal. 
 Non-vanishing of syzygies can be obtained from the general result of M. Green and R. Lazarsfeld \cite[Appendix]{Green}, however, 
 the problem is how to decide if between the bounds given by their result and our bounds the Koszul cohomology groups are zero or not.

 		\end{document}